\documentclass{article}

\usepackage{arxiv}
\usepackage{subcaption}

\usepackage{graphicx}
\usepackage[utf8]{inputenc}
\usepackage[T1]{fontenc}
\usepackage{hyperref}
\usepackage{url} 
\usepackage{booktabs}
\usepackage{amsfonts}
\usepackage{nicefrac}
\usepackage{microtype}
\usepackage{amsmath}
\usepackage{amsthm}
\usepackage{listings}
\usepackage{graphics}
\usepackage{natbib}
\usepackage{bbm}
\usepackage{algpseudocode,algorithm,algorithmicx}
\newcommand*\Let[2]{\State #1 $\gets$ #2}
\DeclareMathOperator\erfc{erfc}

\title{Solving Elliptic Equations with Brownian Motion: \\ Bias Reduction and Temporal Difference Learning}
\date{}

\newtheorem{theorem}{Theorem}

\newtheorem{corollary}{Corollary}

\newcommand{\E}{\mathbb{E}}
\newcommand{\p}{\mathbb{P}}
\newcommand{\dO}{\ensuremath{\rho_{\partial \Omega}}}
\newcommand{\ExitCondition}{\text{ExitCondition}}
\newcommand{\gEstimate}{\text{gEstimate}}

\newcommand{\fEstimate}{\text{fEstimate}}
\newcommand{\tEstimate}{\text{tEstimate}}
\newcommand{\xEstimate}{\text{xEstimate}}

\newcommand{\Proj}{\text{Proj}}
\newcommand{\1}{\mathbbm{1}}
\global\long\def\d{\text{d}}
\author{
  Cameron Martin\\
  \texttt{cameron.martin@mail.utoronto.ca} \\
  \AND
  Hongyuan Zhang\\
  \texttt{hongyuaz@andrew.cmu.edu} \\
  \AND
   Julia Costacurta\\
  \texttt{jcostac@stanford.edu} \\
  \AND
  Mihai Nica\\
  \texttt{nicam@uoguelph.ca} \\
 \AND
  Adam R Stinchcombe\\
  \texttt{stinch@math.toronto.edu}\\
}
\begin{document}
\date{}
\maketitle

\begin{abstract}
The Feynman-Kac formula provides a way to understand solutions to elliptic partial differential equations in terms of expectations of continuous time Markov processes. This connection allows for the creation of numerical schemes for solutions based on samples of these Markov processes which have advantages over traditional numerical methods in some cases. However, na\"{i}ve numerical implementations suffer from statistical bias and sampling error. We present methods to discretize the stochastic process appearing in the Feynman-Kac formula that reduce the bias of the numerical scheme. We also propose using temporal difference learning to assemble information from random samples in a way that is more efficient than the traditional Monte Carlo method.
\end{abstract}


\section{Introduction}

The ability to compute numerical solutions to partial differential equations (PDEs) has proven tremendously important for applications in science and engineering. Many popular numerical schemes, like finite element or finite difference methods, rely on discretizing space and reducing the PDE to a finite dimensional system that can be solved. This works well in many situations. However, in high dimensions and/or in domains that are highly irregular, discretizing space (particularly near the boundary) can be an error-prone and computationally expensive operation. It follows that a promising alternative would forgo spatial discretization altogether. 

We first present a method based on the Feynman-Kac formula to avoid the problem of discretizing space. This famous formula, more traditionally used for theoretical results, connects solutions of PDEs to continuous time Markov processes. The advantage of using this formula is that in our numerical method, we discretize the \emph{time} variable of the Markov processes, but leave the spatial dimensions continuous (up to machine precision). Many authors have proposed methods inspired by the Feynman-Kac formula to solve differential equations, all using some variant of the traditional Monte Carlo method and the walking on spheres technique for simulating Brownian motion~\citep{booth_1981,booth_1982,delaurentis_romero_1990,buchmann2003dirichlet,hwang_mascagni_given_2003,janson2006blackscholes,pauli_gantner_arbenz_adelmann_2015,zhou_cai_2016,zhou_cai_2019}. Their implementations encounter two main difficulties: statistical bias and sampling error.

Statistical bias is a primary challenge with Monte Carlo methods. The manner in which the time variable is discretized obscures some of the underlying behaviour of a sample path of Brownian motion. In particular, a na\"{i}ve discretization results in a systematic overestimation of exit times which introduces statistical bias in the computed solution. Many authors have investigated ways of reducing or eliminating this bias in exit time estimation~\citep{broadie_glasserman_kou_1997,gobet_menozzi_2010,primozic_2011}. Broadie et al.~\citeyearpar{broadie_glasserman_kou_1997} and Gobet et al.~\citeyearpar{gobet_menozzi_2010} suggest using a boundary correction method in which the boundary is artificially shifted a distance in the inward normal direction which depends on the time-step and diffusion coefficient of the process. Primo\v{z}i\v{c}~\citeyearpar{primozic_2011} uses the distribution of the minimum (equivalently, maximum) of a Brownian bridge to estimate one-dimensional boundary passage between time-steps. For reasons that will become clear in section~\ref{sec:exit_conditions}, we term these methods ``bubble wrap'' and ``max-sampling''. We provide additional numerical evidence that the bubble wrap correction is effective. We also generalize max-sampling to $n$ dimensions, and provide further numerical evidence that it will eliminate bias. Both of these corrections are easy to implement and worthwhile to include in any numerical simulation of Brownian motion in a region with boundaries.

Sampling error is another challenge with Monte Carlo methods, which we address with ideas from machine learning. In particular, we use temporal difference learning (TDL) to aggregate information from sample paths. This paradigm, an instance of the wider field of reinforcement learning, has several advantages over the simple Monte Carlo method and has been used to great success by the machine learning community~\citep{suttonbook}. Importantly, this paradigm is also amenable to the use of deep neural networks or other parametrized functions as a basis for a solution. Recent interesting deep learning methods select an objective function from the differential equation directly, analogous to the traditional finite difference and finite element methods~\citep{lagaris_likas_fotiadis_1998,weinan_han_jentzen_2017,han_jentzen_weinan,sirignano_spiliopoulos_2018,raissi_perdikaris_karniadakis_2019,karumuri_tripathy_bilionis_panchal_2020,raissi,raissi_karniadakis_2018,raissi_yazdani_karniadakis,raissi_2018,weinan_yu_2018,zhu_zabaras_koutsourelakis_perdikaris_2019,nabian_meidani_2019}, contrasting the probabilistic approach we study. Han et al.~\citeyearpar{han_nica_stinchcombe} use a deep learning probabilistic method to solve elliptic differential equations, but they did not consider bias in their study. In this paper, we study the effect of bias and sampling error in the TDL paradigm using Chebyshev polynomials as basis functions.

In section~\ref{sec:problem_setting}, we introduce the problem we aim to solve, along with the input requirements of our method and introduce some notation. In section~\ref{sec:monte_carlo_method}, we describe the traditional Monte Carlo method for solving elliptic PDEs, along with various subroutines which will reduce or even eliminate bias from this method. Finally, in section~\ref{sec:temporal_difference_learning}, we describe a method which uses TDL instead of Monte Carlo while retaining the bias reduction subroutines introduced in section~\ref{sec:monte_carlo_method}.

\section{Problem Setting}
\label{sec:problem_setting}

We assume we are on a connected domain $\Omega \subset \mathbb{R}^d$ which has a smooth boundary $\partial \Omega$ and we want to solve elliptic partial differential equations. We will study in detail two prototypical examples for a scalar unknown $u : \Omega \to \mathbb{R}$, namely:
$$ \Delta u(\vec x) =f(\vec x),\quad\vec x \in \Omega,$$
where $f: \Omega \to \mathbb{R}$ and with boundary condition
$$ u(\vec x) = g(\vec x),\quad\vec x\in \partial \Omega,$$
where $g : \partial \Omega \to \mathbb{R} $.

Our method is quite general and will obtain a numerical solution to the problem from the following inputs:
\begin{enumerate}
\item A signed distance function $\dO : \mathbb{R}^d \to \mathbb{R}$, which gives the signed distance to the boundary  $\dO(\vec x) = \pm||\vec{x} - Proj_{\partial \Omega}(\vec x)||$ for all $\vec{x}$. $\dO(\vec x)$ is negative inside $\Omega$, zero on $\partial \Omega$, and positive outside $\Omega$.\\
\item An extension of the boundary data by $g: \mathbb{R}^d \to \mathbb{R}$ which extends the boundary data to all of $\mathbb{R}^d$ in such a way that $g(\vec x) - g( Proj_{\partial \Omega}(\vec x) ) = O(||\vec x -  Proj_{\partial \Omega}(\vec x) ||)$. We abuse notation and often use $g$ to denote this function too.
\end{enumerate}

Our method applies more generally to many types of second order elliptic equations $L[u] = f$, and will be well suited for use in solving parabolic equations of the form $u_t = L[u]+f$. However, to simplify the exposition, we will focus on the simplest case where $L = \Delta$, Poisson's equation.

\subsection{Feynman-Kac Formula: Theory}

The Feynman-Kac formula gives an exact solution to the PDE in terms of an expectation of a Brownian path:

\begin{equation}
 u(\vec{x}) = \E \left[  g(\vec B(T))-\frac12 \intop_{0}^T f(\vec B(t))~\d t~\bigg| ~{\vec{B}(0) = \vec{x}} \right],\label{eq:FeynmanKac}
\end{equation}
where $\vec B(t)$ is a Brownian motion, and $T = \inf\{ t> 0 : \vec B(t) \notin \Omega\}$ is the first exit time of the Brownian motion from the domain $\Omega$. This formula is the inspiration for each of the algorithms we will present.

\subsection{Tangent Plane Approximation: $B_{\perp}$ and $\vec{B}_\parallel$} \label{sec:tangent_plane}

In the case that $\partial \Omega$ is a plane,  it is useful to decompose the $n$-dimensional Brownian motion $\vec B(t)$ into two parts: a $1$-dimensional component $B_{\perp} = \dO(\vec B(t))$ that represents the component of the Brownian motion perpendicular to the boundary $\partial \Omega$, and an $d-1$ dimensional component $\vec B_\parallel$ that represents the remaining directions which are parallel to $\partial \Omega$.

By the properties of Brownian motion, $B_{\perp}$ will be a $1$ dimensional Brownian motion and $\vec B_\parallel$ will be an independent $d-1$ dimensional Brownian motion. The hitting time $T$ in the Feynman-Kac formula in this case is the classic barrier problem $T = \inf\{s: B_\perp(s) > 0\} $ of a 1 dimensional Brownian motion. Since $\vec B_\parallel$ is independent of $T$ in this case, the location of exit $\vec B_\parallel(T)$ can be well understood.

For a more general surface $\partial \Omega$, the above independence and relation to one-dimensional Brownian motion is not as clear. Also, the definition of $\vec B_\parallel(T)$ is not obvious. However, for the purposes of the Feynman-Kac formula, and assuming that $\partial \Omega$ is sufficiently smooth, when $\vec B(t)$ is close to $\partial \Omega$, and when we look over a short enough time interval, $\partial \Omega$ is well approximated by the tangent plane to the surface $\Omega$ at the point $\Proj_{\partial \Omega}\vec B(t)$. We will use this approximation by a tangent plane to develop our numerical methods below.

Corners in $\partial \Omega$ are not well approximated this tangent plane approximation. However, corners can be understood in terms of a first hitting time to more than one plane, so a more complicated version of our analysis that takes into account what happens for multiple planes could apply. We do not explore this here, but leave this issue for future work.

\subsection{Example Problems}
We will demonstrate our methods in the case $d=2$ on the unit disk  $\Omega = {x_1^2 + x_2^2 < 1} \subset \mathbb{R}^2$. The signed distance function is $\dO(\vec{x}) =  \pm \sqrt{|1- x_1^2 +x_2^2}|$ and $\1_\Omega = 1\{x_1^2 + x_2^2 < 1\}$

We will look at two problems for this domain. Both problems have simple exact solutions that allow us to investigate the error of our methods.

\qquad Dirichlet Problem: \quad
$f \equiv 0$ and $g(\vec x) = \1\{x_2 >0\}$ with exact solution
$u(\vec x) = \frac12 + \frac{1}{\pi}\mathrm{arctan}\left(\frac{2x_2}{1-x_1^2-x_2^2}\right).$

\qquad Poisson Problem: \quad
$f \equiv 1$ and $g \equiv 0$ with exact solution
$u(\vec x) = \frac14(x_1^2+x_2^2-1).$

\section{Monte Carlo Method}
\label{sec:monte_carlo_method}

If one could sample Brownian paths $B(\cdot)$ exactly, a simple estimate for $u(\vec{x}_0)$ using the Feynman-Kac formula would be obtained by the  following algorithm:

\qquad Step 0. Sample $B(\cdot)$ started from $B(0)=\vec x_0$.

\qquad Step 1. Find the hitting time $T$. 

\qquad Step 2. Find $\intop_0^T f(\vec B(t))~\d t$.

\qquad Step 3. Find $g(\vec B_T)$.

Adding these together as in Eq.~\eqref{eq:FeynmanKac} would give us an unbiased estimate for $u(\vec{x}_0)$. Taking an empirical average over $N$ such Brownian motions we get an estimate for $u(\vec{x}_0)$ whose error is typically of size $\sqrt{N}^{-1}$. Averaging over many samples in this way is the basis for the well-known Monte Carlo method. 

In practice, we do not have access to a Brownian path $B(\cdot)$ exactly, and we must estimate Step 1, 2 and 3 by some discretization scheme. Once this has been done, we obtain the Monte Carlo method, Alg.~\ref{alg:monte_carlo}. The parameters to this algorithm are: $N$, the number of samples; $\Delta \tau$, the time-step used for discretizing the Brownian motion; and $\vec{x}_0$, the location at which we estimate the value of $u$.

The algorithm relies on three subroutines which estimate the different parts of the Feynman-Kac formula Eq.~\eqref{eq:FeynmanKac}:
\begin{enumerate}
    \item $\ExitCondition$, which estimates the hitting time $T$ to within a $\Delta\tau$-length interval;
    \item $\fEstimate$, which estimates the contribution from $\intop_0^T f(\vec B(t))~\d t$;
    \item $\gEstimate$, which estimates the contribution from  $g(\vec B_T)$.
\end{enumerate}

By changing the subroutines ExitCondition, fEstimate, and gEstimate, we can obtain variations on the basic Monte Carlo method. The most na\"{i}ve choices for these functions are consistent (in that convergence occurs as $\Delta \tau \to 0$), but biased in that there are systematic errors that do not go to zero as the number of samples goes to infinity. Good choices for these subroutines can improve the accuracy of the Monte Carlo method considerably.

In the algorithm, $\vec B_{\mathrm{old}}$ and $\vec B_{\mathrm{new}}$ always represent successive positions of a Brownian motion sampled time $\Delta \tau$ apart~\textemdash~we can think of this as $\vec B_{\mathrm{old}}=\vec B(0)$ and $\vec B_{\mathrm{new}}=\vec B(\Delta \tau)$. What the Brownian motion does at the intermediate times $t \in (0,\Delta \tau)$ is not sampled by the algorithm but is nevertheless relevant to the Feynman-Kac formula. For example, the hitting time $T$ might occur during an intermediate time! Good choices for ExitCondition, fEstimate, and gEstimate take this point of view and try to account for this interval $(0,\Delta \tau)$. Conditioned on the values $\vec B_{\mathrm{old}}$ and $\vec B_{\mathrm{new}}$, the Brownian motion is a Brownian bridge, which means we know and can exploit the law of the process on $(0,\Delta \tau)$.

In the following sections, we will describe and analyze some options for ExitCondition, fEstimate, and gEstimate. In Fig.~1, the pointwise error for the Monte Carlo method applied to the two test problems, Dirichlet and Poisson, is shown for three choices for ExitCondition: na\"{i}ve, bubble-wrap, and max-sampling, and the na\"{i}ve and corrected versions of gEstimate and fEstimate. A large time-step $\Delta \tau=0.1$ is used to accentuate the bias. By reducing the bias, non-na\"{i}ve choices for the subroutines dramatically reduce the method's error.

\begin{algorithm}
  \caption{Monte Carlo method: provides an estimate for $u(\vec{x}_0)$. Several variations are possible by varying the ExitCondition, fEstimate, and gEstimate functions.}
  \begin{algorithmic}[1]
    \Require{Functions $\fEstimate,\gEstimate, \ExitCondition$}
    \Statex
    \Function{MonteCarlo}{$N, \Delta \tau, \vec{x}_0$, ExitCondition, gEstimate,  fEstimate}
      \Let{$U$}{$0$}
      \For{$i \in \{1,2,\ldots,N\}$}
      \Let{$W$}{$0$}
      \Let{$\vec B_{\mathrm{old}}$}{$\vec x_0$}
      \Let{$\vec B_{\mathrm{new}}$}{$\vec x_0$}
        \While{NOT $\ExitCondition(\vec B_{\mathrm{old}},\vec B_{\mathrm{new}}, \Delta \tau)$}
         
         \Let{$\vec B_{\mathrm{old}}$}{$\vec B_{\mathrm{new}}$}
         \Let{$\vec B_{\mathrm{new}}$}{$\vec B_{\mathrm{old}} + \sqrt{\Delta \tau} \mathcal{N}(\vec{0},I_d)$}
         \Let{$W$}{$W - \frac12 \fEstimate(\vec{B}_{\mathrm{old}}, \vec{B}_{\mathrm{new}}, \Delta \tau)$}
        \EndWhile
      \Let{$W$}{$W + \gEstimate(\vec{B}_{\mathrm{old}}, \vec{B}_{\mathrm{new}}, \Delta \tau)$}
      \Let{$U$}{$U + \frac{1}{N}W$}
      \EndFor
      \State \Return{$U$}
    \EndFunction
  \end{algorithmic}
\label{alg:monte_carlo}
\end{algorithm}

\subsection{$\ExitCondition$}
\label{sec:exit_conditions}
The most na\"{i}ve exit condition for the Monte Carlo scheme is:
$$\text{na\"{i}ve: }\ExitCondition(\vec B_{\mathrm{old}},\vec B_{\mathrm{new}} ,\Delta \tau) = \1{\{\vec B_{\mathrm{new}}} \notin \Omega\}$$

It turns out that this exit condition has a systematic bias of order $O(\sqrt{\Delta \tau})$ to overestimate the exit time $T$, shown in Cor.~\eqref{cor:T_estimate} below.

\subsubsection{Discretization Bias of the Na\"{i}ve Exit Condition}

 Whenever two subsequent steps remain inside the disk, there is a possibility that an exit nonetheless occurred between those two steps.  Hence the na\"{i}ve exit condition will systematically overestimate the exit time $T$. This can lead to systematic bias in estimates to the integral $\intop_0^T f(B(t))~\d t$ (for example if $f$ is always positive). Overestimating $T$ also means the distance the Brownian motion travels from its starting point $B(0)$ to its exit point $B(T)$ will be overestimated. Hence the value of $g(B(T))$ will be sampled at locations further from the starting point $\vec x_0$ than the true exit location. This will lead to an underestimate at locations near the boundary where $g$ is larger than average and an overestimate at locations near the boundary where $g$ is smaller than average.
 
 The following theorems quantify this ``overestimation-of-$T$'' bias in the case that $\partial \Omega$ is a plane using the decomposition of the distance to the plane as in section~\ref{sec:tangent_plane}. When $\partial \Omega$ is well approximated by a tangent plane, we should expect these error estimates to also hold. 

\begin{theorem}
\label{thm:T_Estimate}
Suppose that $\partial \Omega$ is a plane, and the Brownian motion $\vec B$ is decomposed into its perpendicular and parallel directions $B_\perp$ and $\vec B_\parallel$ as in section~\ref{sec:tangent_plane}. Let $T=\inf\left\{ t:B_{\perp}(t)\geq 0 \right\} $ be the true time
of exit and let $T^{\Delta\tau}=\min_{n\in\mathbb{N}}\left\{ n\Delta\tau:B_{\perp}(n\Delta\tau)\geq 0 \right\} $
be the first time that the discretization using time-steps of size $\Delta\tau$ observes an exit. Then $T^{\Delta \tau} -  T$ is of order $\Delta \tau$ and the difference converges in distribution according to
\begin{equation}
\frac{T^{\Delta\tau}-T}{\Delta \tau} \Rightarrow U + \min\{ k\geq 0: W(U+k) >0 \},
\end{equation}
in which $U$ is a uniform $(0,1)$ random variable and $W(\cdot)$ is an independent standard Brownian motion.
\end{theorem}

\begin{proof}
By the decomposition of section~\ref{sec:tangent_plane}, the hitting time $T$ is precisely the hitting time of a 1 dimensional random walk. The result then follows by Theorem 1 of~\citep{1708.04356}. 
\end{proof}

\begin{corollary}
\label{cor:T_estimate}
With the same assumptions and definitions as in theorem~\ref{thm:T_Estimate}, assume also that the function $f$ is bounded by $||f||_\infty$ and that the function $g$ is Lipschitz with Lipschitz constant $||g||_{\text{Lip}}$. Then the error in estimating $T$ leads to an error in $\intop_0^T f(\vec B(t))~\d t$ and $g(\vec B(T))$ of sizes
\begin{equation}\left| \intop_0^{T^{\Delta\tau}} f(\vec B(t))~\d t -\intop_0^{T} f(\vec B(t))~\d t \right| \leq O(\Delta \tau) ||f||_{\infty},
\end{equation}
and
\begin{equation}\left|g(\vec B(T^{\Delta \tau})) - g(\vec B(T)) \right|\leq O(\sqrt{\Delta \tau}) ||g||_{\text{Lip}}.
\end{equation}

\end{corollary}
\begin{proof}
The first result follows immediately from theorem~\ref{thm:T_Estimate}. For the second, notice that on the boundary $\partial \Omega$, by the decomposition in section~\ref{sec:tangent_plane}, that  $g(\vec B(t))$ depends only on $\vec B_\parallel(t)$. But $\vec B_\parallel(\cdot)$ is independent of $B_\perp(\cdot)$. Hence, by independence and by Brownian scaling, since $T^{\Delta \tau} - T = O(\Delta \tau)$ from theorem~\ref{thm:T_Estimate}, we will have that $||\vec B_\parallel(T^{\Delta \tau}) - \vec B_\parallel(T)|| = O(\sqrt{\Delta \tau})$ and the result follows.
\end{proof}

\begin{figure}[!htbp]
  \centering
  \includegraphics{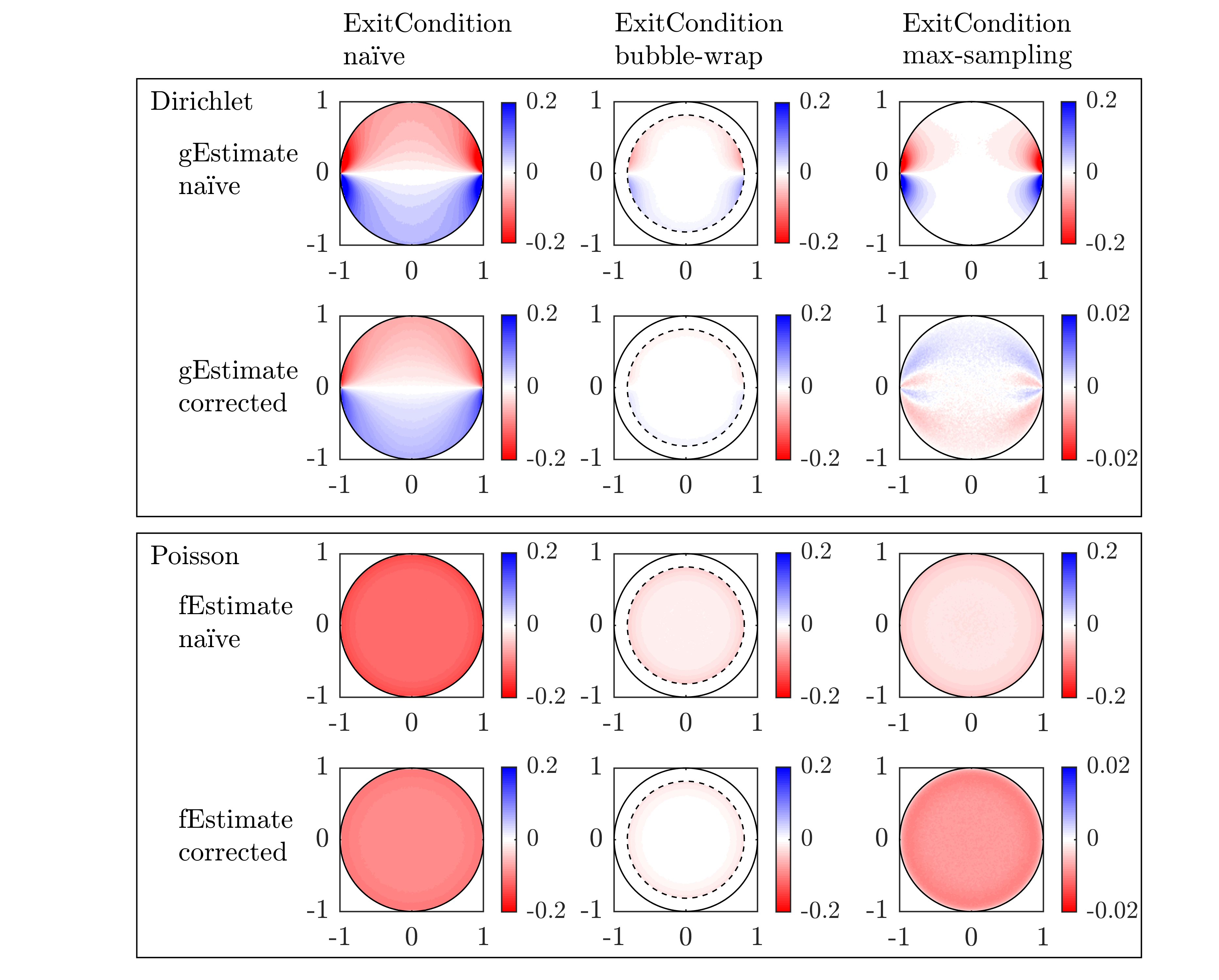}
  \caption{The bias in the Monte Carlo method. In the upper plots, the boundary data is $g(x_1,x_2)=\1\{x_2\geq0\}$, the Dirichlet problem. In the lower plots, the forcing is $f(x_1,x_2)=1$, the Poisson problem. In each column, a different exit condition is used: na\"{i}ve, bubble-wrap, and max-sampling. The two rows for each problem show different $\gEstimate$ and $\fEstimate$ methods: na\"{i}ve and corrected. For the bubble-wrap exit condition, the bias is not computed within $b$ of the boundary (shown with the dashed line). In all cases $\Delta \tau=0.1$. For the Dirichlet problem, a symmetry in the problem results in zero bias along the line $x_2=0$. In the upper half-disk away from $x_2=0$ where the boundary data is one, the random walk method systematically underestimates the solution, while in the lower half-disk away from $x_2=0$ where the boundary data is zero, the solution is an overestimate. The bias is most significant near the discontinuities in the boundary data. For the Poisson problem, the solution is always an underestimate. For each problem, the bubble-wrap and max-sampling exit conditions significantly reduce the bias. Note the factor of 10 reduction in scale when max-sampling is used with either estimate correction method.}
  \label{fig:bias}
\end{figure}

\subsubsection{Bubble-Wrap Exit Condition}

The discretization error in the na\"{i}ve exit condition is due to the discrete process systematically under counting collisions with the boundary. One way to reduce this bias is to counteract this undercounting by systematically increasing the collisions of the discrete process with the boundary. 

The simplest and most easily implemented way to reduce some of this bias is to ``bubble wrap'' each walker: count a walker as having ``crossed'' the boundary if it ever comes within some threshold distance $b > 0$ from the barrier (as opposed to only counting walkers that cross the boundary outright). Morally speaking, this threshold distance $b$ is supposed to account for walkers whose continuous time trajectories had crossed the boundary in between the two discrete samples $\vec B_{\mathrm{old}},\vec B_{\mathrm{new}}$. The exit condition to be used in the Monte Carlo method for this is
$$\text{bubble-wrap: }\ExitCondition(\vec B_{\mathrm{old}},\vec B_{\mathrm{new}}, \Delta \tau) = \1\{\dO(\vec B_{\mathrm{new}}) > -b\}$$

Since the bubble radius $b$ is supposed to account of the maximum of a Brownian motion on a time increment $\Delta \tau$, we should choose $b$ to be on the same order of this maximum: namely order $\sqrt{\Delta \tau}$. There are some theoretical reasons to believe that
$$ b = \frac{| \zeta(\frac12) |}{\sqrt{2\pi}} \sqrt{\Delta \tau} \approx 0.583 \sqrt{\Delta \tau} $$
is a good choice. This is because in the limit that $\Delta \tau \to 0$, the expected height difference between the discrete and continuous walks at the moment they are first observed to cross the boundary is $\frac{| \zeta(\frac12) |}{\sqrt{2\pi}} \sqrt{\Delta \tau}$.  The next theorem makes this more precise.

\begin{theorem}
Suppose that $\partial \Omega$ is a plane. Let $T=\inf\left\{ t:B_{\perp}(t)\geq 0 \right\} $ be the true time of exit and let $T^{\Delta\tau}=\min_{n\in\mathbb{N}}\left\{ n\Delta\tau:B_{\perp}(n\Delta\tau)\geq 0 \right\} $ be the first time that the discretization using time-steps of size $\Delta\tau$ observes an exit. Then the typical distance to the boundary $\dO$ observed at time $T^{\Delta\tau}$ is
$$\lim_{\Delta \tau \to 0} \frac{\E\left[  \dO(\vec B(T^{\Delta \tau})) \right]} {\sqrt{\Delta \tau}} = \frac{| \zeta(\frac12) |}{\sqrt{2\pi}}.$$

Moreover, if we define $M = \sup_{0<t<1} \dO( \vec B(t))$ to be the maximum distance over the time interval $t\in [0,1]$ and let $M^{\Delta \tau} =\max_{1\leq n \leq (\Delta \tau)^{-1}} \dO(\vec B( n \Delta \tau))$ be the maximum sampled over the discrete grid, then for any $x \in \mathbb{R}$,
$$
\lim_{\Delta \tau \to 0} \p\left(M^{\Delta \tau} + \frac{| \zeta(\frac12) |}{\sqrt{2\pi}} \sqrt{\Delta \tau} < x\right) = \p(M < x).
$$

\end{theorem}
\begin{proof}
By the decomposition of section~\ref{sec:tangent_plane}, the problem is reduced to the overshoot of a 1 dimensional random walk hitting a barrier. The result then follows from Theorem 1 and Proposition 1 of~\citep{1708.04356}.
\end{proof}

Both of the statements of the theorem suggest that, on average, the discretized process tends to underestimate the maximum of the Brownian motion by $\frac{| \zeta(\frac12)|}{\sqrt{2\pi}} \sqrt{\Delta \tau}$, which makes this a natural choice for the bubble radius. In Fig.~\ref{fig:bubbleradius}, an estimate of the bias is computed for a range of bubble radii. The choice of $\frac{| \zeta(\frac12)|}{\sqrt{2\pi}}\sqrt{\Delta \tau}$ does a good job of eliminating the bias for both the Dirichlet and Poisson problems. 

\begin{figure}[!htbp]
  \centering
  \includegraphics{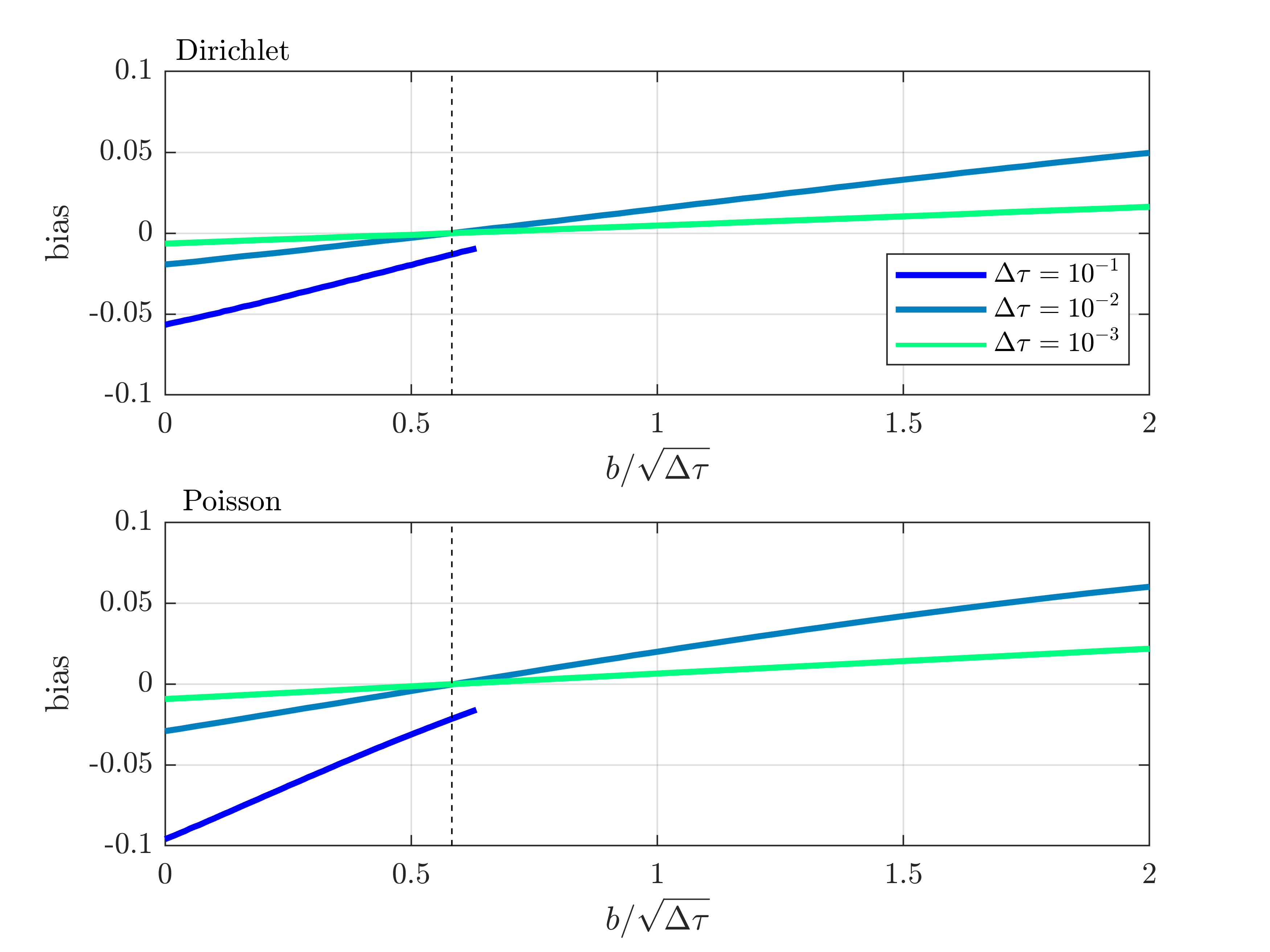}
  \caption{The bias in estimating $u(\vec{x})$ at $\vec{x}=(0.8\cos(\pi/3),0.8\sin(\pi/3))^T$ as a function of the bubble parameter $b$ for the Dirichlet and Poisson problems. Three values of $\Delta \tau=10^{-1},10^{-2},10^{-3}$ are shown with three colours. The dashed line shows $b/\sqrt{\Delta \tau}=|\zeta(1/2)|/\sqrt{2\pi}$, which very nearly eliminates the bias. Note that $b=0$ corresponds to the na\"{i}ve exit condition. The bias was estimated for each value of $b$ using $2^{24}$ samples.}
  \label{fig:bubbleradius}
\end{figure}

\subsubsection{Max-Sampling Exit Condition}
In the max-sampling ExitCondition, we further attempt to control the ``late exit'' discretization bias by a more accurate (and more complicated) exit condition. Here $\ExitCondition$ is a \emph{random} function: for any fixed vectors $\vec B_{\mathrm{old}},\vec B_{\mathrm{new}}$ it will sometimes evaluate to True and sometimes evaluate to False. The idea is that, conditioned on the values of $\dO(\vec B_{\mathrm{old}}))$ and $\dO(\vec B_{\mathrm{new}})$, the Brownian motion $B_\perp = \dO(\vec B(t))$ is approximately a Brownian bridge with $B_\perp(0)=\dO(\vec B_{\mathrm{old}})$ and $B_\perp(\Delta \tau) =\dO( \vec B_{\mathrm{new}})$. This is exactly true in the case that $\partial \Omega$ is a plane. By the reflection principle, the maximum of a one-dimensional Brownian bridge with $B(0) = 0$ and $B(\Delta \tau) = x$, is distributed like
$$M = \max_{t \in [0,\Delta \tau]} \left\{B(t)~|~B(0)=0, B(\Delta \tau)=x\right\} \stackrel{d}{=} \frac12( x + \sqrt{x^2 + 2\Delta \tau   E}),$$
in which $E\sim\textrm{Exp}(1)$ is an exponential random variable independent of everything else. Thus the maximum of $B_\perp$ can be simulated by this formula. This gives the exit condition
\begin{align*}
    \text{max-sample: }\ExitCondition(\vec B_{\mathrm{old}},\vec B_{\mathrm{new}}, \Delta \tau) &= \1\left\{ \frac12\left( \Delta \rho + \sqrt{(\Delta \rho)^2 + 2 \Delta \tau E}\right) > |\dO(\vec{B}_{\mathrm{old}})|\right\},\\
    \Delta \rho &:= \dO(\vec{B}_{\mathrm{new}}) - \dO(\vec{B}_{\mathrm{old}}).
\end{align*}
Incidentally, this exit condition will be satisfied automatically if it happens that $\dO(\vec{B}_{\mathrm{new}}) > 0$, i.e. $\vec{B}_{\mathrm{new}} \notin \Omega$. 

\subsection{Exit Time and Location Estimates}
\label{sec:time_loc_estimates}

Once we've determined that a Brownian motion has exited in a time interval of length $\Delta\tau$, we have to estimate when and where the exit occurred. From these estimates, we can obtain estimates for $\intop_0^{T} f(B(t))~\d t$ and for the value of the boundary condition at the location of exit. The na\"{i}ve choice for the exit time estimate is
$$\text{na\"{i}ve: }\tEstimate(\vec{B}_{\mathrm{old}}, \vec{B}_{\mathrm{new}}, \Delta \tau) = \Delta \tau.$$

Taking the midpoint of the time interval gives a slightly better tEstimate,
$$\text{na\"{i}ve+: }\tEstimate(\vec{B}_{\mathrm{old}}, \vec{B}_{\mathrm{new}}, \Delta \tau) =\frac{\Delta \tau}{2}.$$ 

However, we can be more precise than this by using the known distribution of the first exit time of a Brownian bridge. In the case that $\partial \Omega$ is a plane, the distribution of the exit time $T$ of the Brownian motion is given by the normalized intensity found in theorem 2 of~\citep{hieber_2013} (normalization is required when $\vec{B}_{\mathrm{new}} \in \Omega$). Translating this into our notation, the probability distribution for this time is
\begin{equation}
    f_{T}(t) = \frac{|\dO(\vec B_{\mathrm{old}})|}{\sqrt{2\pi}t^{3/2}\sqrt{1 - t/\Delta\tau}}\exp\left(\frac{(\Delta\rho)^2}{2\Delta\tau}-\frac{\dO(\vec B_{\mathrm{new}})^2}{2(\Delta\tau - t)} - \frac{\dO(\vec B_{\mathrm{old}})^2}{2t}\right).
    \label{eqn:exittimeintensity}
\end{equation}

Because of the independence between the direction normal to the plane $\partial \Omega$ and the orthogonal $d-1$ dimensional boundary, the exit location distribution is then the distribution of the Brownian bridge at time $T$. The distribution of the exit location conditional on the exit time $T$ is therefore
\begin{equation}
    \vec{x}_\mathrm{exit} \sim \Proj_{\partial\Omega}\left(\mathcal{N}\left(\left(1 - \frac{T}{\Delta\tau}\right)\vec{B}_{\mathrm{old}} + \frac{T}{\Delta\tau}\vec{B}_{\mathrm{new}}, \frac{T(\Delta\tau - T)}{\Delta\tau}I_d\right)\right).
    \label{eqn:Xexitdist}
\end{equation}

Ideally, we would sample $T$ and $\vec{x}_\mathrm{exit}$ as our exit time and location estimates. Unfortunately, the distribution of $T$ is difficult to sample efficiently and varies with $\dO(\vec{B}_\mathrm{old})$ and $\dO(\vec{B}_\mathrm{new})$. To circumvent this issue, we can exploit a connection to the distribution of local times of a Brownian bridge to obtain an approximation to the \emph{expected} exit time $\mathbb{E}\left[T\right]$, namely,

\begin{equation} 
\frac{1}{1 + \Delta \tau \left(\Delta\rho\right)^{-2}} \leq  \frac{\mathbb{E}\left[T\right]}{ (1-\lambda)\Delta \tau} \leq 1, \label{eqn:exit_time_ineq}
\end{equation}
in which
$$
\Delta\rho := \dO(\vec B_{\mathrm{new}})
- \dO(\vec B_{\mathrm{old}}),
$$
and
\begin{equation}
    \lambda := \frac{|\dO(\vec B_{\mathrm{new}})|}{|\dO(\vec B_{\mathrm{old}})| + |\dO(\vec B_{\mathrm{new}})|}. \label{eqn:lambda_defn}
\end{equation}
This inequality is proven for $\vec{B}_\mathrm{new} \notin \Omega$ in Appendix \ref{A:exit_time_ineq}. By the squeeze theorem,  Eq.~\eqref{eqn:exit_time_ineq} shows that as $\Delta \tau \to 0$, that
\begin{equation}
    \mathbb{E}\left[T\right] \to (1-\lambda)\Delta \tau. \label{eqn:T_estimate}
\end{equation}
We use this approximation as an easy means to approximate $T$,
$$\text{Corrected: }\tEstimate(\vec{B}_{\mathrm{old}},\vec{B}_{\mathrm{new}},\Delta\tau) =  (1-\lambda)\Delta\tau.$$

$(1-\lambda)\Delta\tau$ is also the time at which the linear interpolant between the start and end point of the Brownian bridge would hit the boundary. Given this approximation for $T$, we can also simply approximate $\vec{x}_{\mathrm{exit}}$ to be its mean from Eq.~\eqref{eqn:Xexitdist},
$$\text{Corrected: }\xEstimate(\vec{B}_{\mathrm{old}},\vec{B}_{\mathrm{new}}) =  (1-\lambda)\vec{B}_{\mathrm{new}} +\lambda \vec{B}_{\mathrm{old}}.$$

\subsection{fEstimate}
The function $\fEstimate$ provides an estimate for $\intop_{0}^{\min(T,\Delta \tau)} f(B(t))~\d t$ given $\vec B(0) = \vec B_{\mathrm{old}}$ and $\vec B(\Delta \tau) = \vec B_{\mathrm{new}}$. The na\"{i}ve choice for fEstimate is
$$\text{na\"{i}ve: }\fEstimate(\vec{B}_{\mathrm{old}}, \vec{B}_{\mathrm{new}}, \Delta \tau) = \Delta \tau f(\vec{B}_{\mathrm{old}}),$$ 
which already achieves error $O(||f||_\text{Lip}\Delta \tau)$ when $\Delta \tau < T$. The trapezoid rule gives a slightly better fEstimate,
$$\text{na\"{i}ve+: }\fEstimate(\vec{B}_{\mathrm{old}}, \vec{B}_{\mathrm{new}}, \Delta \tau) =\frac{\Delta \tau}{2} \left( f(\vec{B}_{\mathrm{old}}) + f(\vec{B}_{\mathrm{new}}) \right).$$ 

However, if the hitting time $T$ happens in between $\vec{B}_{\mathrm{old}}$ and $\vec{B}_{\mathrm{new}}$ (i.e. on the time interval $(0,\Delta \tau)$, then the resulting integral $\intop_{0}^{\min(T,\Delta \tau)} f(B(t))~\d t$ integrates over a time interval shorter than $\Delta \tau$. In this situation, the estimate for the integral of $f$ needs to be cut short. Using the estimate for the exit time $T$ from section~\ref{sec:time_loc_estimates},
$$\text{Corrected: }\fEstimate(\vec{B}_{\mathrm{old}}, \vec{B}_{\mathrm{new}}, \Delta \tau) = (1-\lambda) \Delta \tau f(\vec{B}_{\mathrm{old}}) \text{ if } T < \Delta \tau,$$
in which $\lambda$ is as in Eq.~\eqref{eqn:lambda_defn}. This correction formula can be easily implemented in the Monte Carlo method by adding a correction of $-\lambda \Delta \tau f(\vec{B}_{\mathrm{old}}) \text{ if } T < \Delta \tau$ to the accumulated estimate after $\ExitCondition$ is triggered.

\subsection{gEstimate}
\label{gEstimate}

The function $\gEstimate$ provides an estimate of the value of the boundary condition at the location of the exit of the Brownian motion. The na\"{i}ve estimate is simply to evaluate $g$ at the end point $\vec B_{\mathrm{new}}$,
$$\text{na\"{i}ve: } \gEstimate(\vec B_{\mathrm{old}}, \vec B_{\mathrm{new}}) = g(\vec B_{\mathrm{new}}).$$

This corresponds to an exit time estimate of $\Delta\tau$, the end of the time interval in which the walker exited, which is clearly an overestimate. According to Cor.~\eqref{cor:T_estimate}, the error in this estimate is $O(\sqrt{\Delta\tau})$. Using the estimate for exit location from section~\ref{sec:time_loc_estimates}, we obtain a better \gEstimate,
$$\text{Corrected: }\gEstimate(\vec{B}_{\mathrm{old}},\vec{B}_{\mathrm{new}}) = g\left( (1-\lambda)\vec{B}_{\mathrm{new}} +\lambda \vec{B}_{\mathrm{old}} \right).$$

\subsubsection{Brownian Root-Finding}

While the corrected gEstimate is clearly superior to the na\"{i}ve gEstimate, it is not clear how its accuracy will scale with the time-step $\Delta\tau$. It would be beneficial to have a strategy which does not require a small time-step to obtain a reasonable gEstimate. With such a strategy, one could enjoy the computational benefits of a large time-step while maintaining the accuracy which comes with a small time-step. This is the motivation behind what we term Brownian root-finding (BRF). This algorithm is to be used when $\vec{B}_{\mathrm{old}} \in \Omega$ and $\vec{B}_{\mathrm{new}} \notin \Omega$ and functions by sampling the Brownian bridge between these two points repeatedly to improve the estimate of the exit location. The term ``root-finding'' reflects the similarities between BRF and the commonly used bisection method of root-finding.

This procedure, much like the max-sampling exit condition, turns the exit time and location (hence gEstimate as well) into \emph{random} functions: for any fixed vectors $\vec{B}_{\mathrm{old}} \in \Omega, \vec{B}_{\mathrm{new}} \notin \Omega$, it will return varying exit times and locations. This algorithm makes use of the fact that when $\vec{B}_{\mathrm{old}} \in \Omega$ and $\vec{B}_{\mathrm{new}} \notin \Omega$, the Brownian motion from $\vec{B}_{\mathrm{old}}$ to $\vec{B}_{\mathrm{new}}$ is a Brownian bridge starting inside the domain and ending outside. The BRF algorithm will iteratively sample from Brownian bridges, honing in on the boundary and terminating with a good estimate for a walker's exit location.

BRF will not be applicable when $\vec{B}_{\mathrm{old}} \in \Omega$ and $\vec{B}_{\mathrm{new}} \in \Omega$, because a Brownian bridge from $\vec{B}_{\mathrm{old}}$ to $\vec{B}_{\mathrm{new}}$ is not guaranteed to exit the domain. 

\begin{figure}[!htbp]
  \centering
  \includegraphics[width = \textwidth]{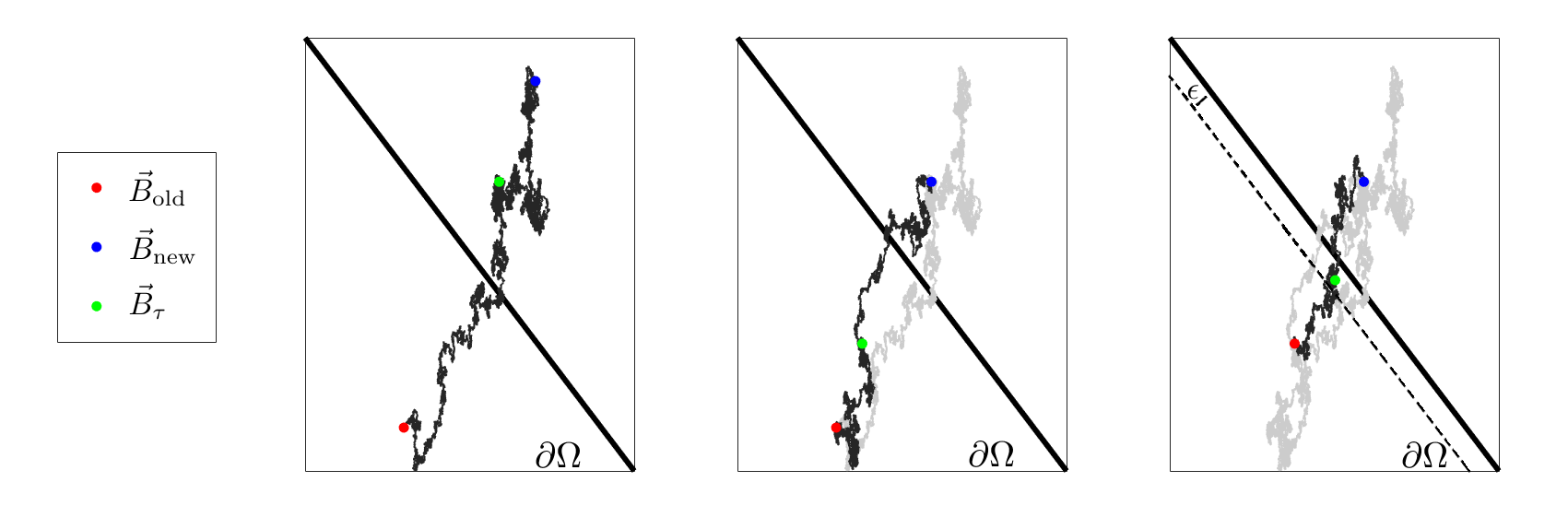}
  \caption{A depiction of the Brownian root-finding algorithm. Given samples $\vec{B}_{\mathrm{old}}$ and $\vec{B}_{\mathrm{new}}$, successive estimates of the exit time and location are computed by sampling Brownian bridges between points on the Brownian path inside and outside of the domain $\Omega$. In the first frame, points $\vec{B}_{\mathrm{old}}$ and $\vec{B}_{\mathrm{new}}$ are respectively inside and outside the domain. The point $\vec{B}_\tau$ is a sample of the Brownian bridge between $\vec{B}_{\mathrm{old}}$ and $\vec{B}_{\mathrm{new}}$. In the first panel, $\vec{B}_\tau$ happened to be outside the domain. In the second frame, we replace $\vec{B}_{\mathrm{new}}$ with the previous frame's $\vec{B}_\tau$ and record a new $\vec{B}_\tau$ as a sample of the Brownian bridge between points $\vec{B}_{\mathrm{old}}$ and $\vec{B}_{\mathrm{new}}$. The process continues in the third frame, terminating when $\vec{B}_\tau$ is within a small distance $\epsilon$ of the boundary. The detailed algorithm can be found in Alg.~\ref{alg:BRF}.}
  \label{fig:BRF}
\end{figure}

The procedure goes as follows. Letting $\vec B_{\mathrm{old}}=\vec B(0)$ and $\vec B_{\mathrm{new}}=\vec B(\Delta \tau)$, sample the Brownian bridge at time $\tau = \theta\Delta\tau$, where $\theta \in (0,1)$ is a parameter of the algorithm. If $\vec{B}_\tau \notin \Omega$, then we know that the Brownian motion exited prior to $\tau$. In this case, we assign $\vec{B}_\tau$ to $\vec{B}_{\mathrm{new}}$, and repeat. If $\vec{B}_\tau \in \Omega$, then it is possible, but not guaranteed, that the Brownian motion exited some time before $\tau$. In this case, we borrow ideas from section~\ref{sec:exit_conditions}, and apply some $\ExitCondition$ to $\vec{B}_{\mathrm{old}}, \vec{B}_\tau$, and $\tau$. If the Brownian motion is not determined to have exited, then we assign $\vec{B}_\tau$ to $\vec{B}_{\mathrm{old}}$ and repeat. If it is determined to have exited, then we are in the situation where $\vec{B}_{\mathrm{old}} \in \Omega$, $\vec{B}_\tau \in \Omega$, and the Brownian motion exited at some intermediate time. In this case, BRF is not applicable, and we simply return exit time and location estimates based on $\vec{B}_{\mathrm{old}},\vec{B}_\tau,$ and $\tau$. The detailed algorithm can be found in Alg.~\ref{alg:BRF}. In effect, BRF converts an exit condition and an exit time/location estimator when $\vec B_{\mathrm{new}}\in\Omega$ into an exit time/location estimator for any $\vec B_{\mathrm{new}}$. This procedure leads to the following gEstimate variation,
$$\text{BRF: }\gEstimate(\vec{B}_{\mathrm{old}},\vec{B}_{\mathrm{new}}) = g\left(\text{BRF}\left(\vec{B}_{\mathrm{old}},\vec{B}_{\mathrm{new}},\theta,\epsilon,\rho_{\partial\Omega},\ExitCondition,\tEstimate,\xEstimate\right)\right).$$

In all implementations, we use the max-sampling $\ExitCondition$ and the corrected $\tEstimate$ and $\xEstimate$. This method will be more computationally intensive than the corrected $\xEstimate$ alone, but it should be more accurate when using large time-steps.

Before we test whether BRF provides an improvement when using large-time steps, we will investigate the impact of BRF's parameter $\theta$. To do this, we simulate a one-dimensional Brownian motion starting at the origin $B(0) = 0$. In this case, there is a known distribution for the first passage time across $a$: $T^a := \inf\limits_{t \geq 0}\left\{B(t) = a\right\}$ is L\'{e}vy distributed with CDF $F(t) = \erfc\left(\frac{|a|}{\sqrt{2t}}\right)$. In Fig.~\ref{fig:BRFtest_theta}, we record empirical exit times past the barrier $a = 1$ using BRF estimates with different values of $\theta$ and compare these to the known distribution. We see that BRF overestimates the exit time. Also, BRF approximates the true CDF best for $\theta$ close to 1 and performs worse than the corrected estimate for $\theta$ close to 0. The difference in the CDF between $\theta = 0.5$ and $\theta = 0.95$ is small, despite $\theta = 0.95$ requiring significantly more samples of Brownian bridges. Therefore, we opt to use $\theta = 0.5$.

\begin{figure}[h]
    \centering
    \includegraphics{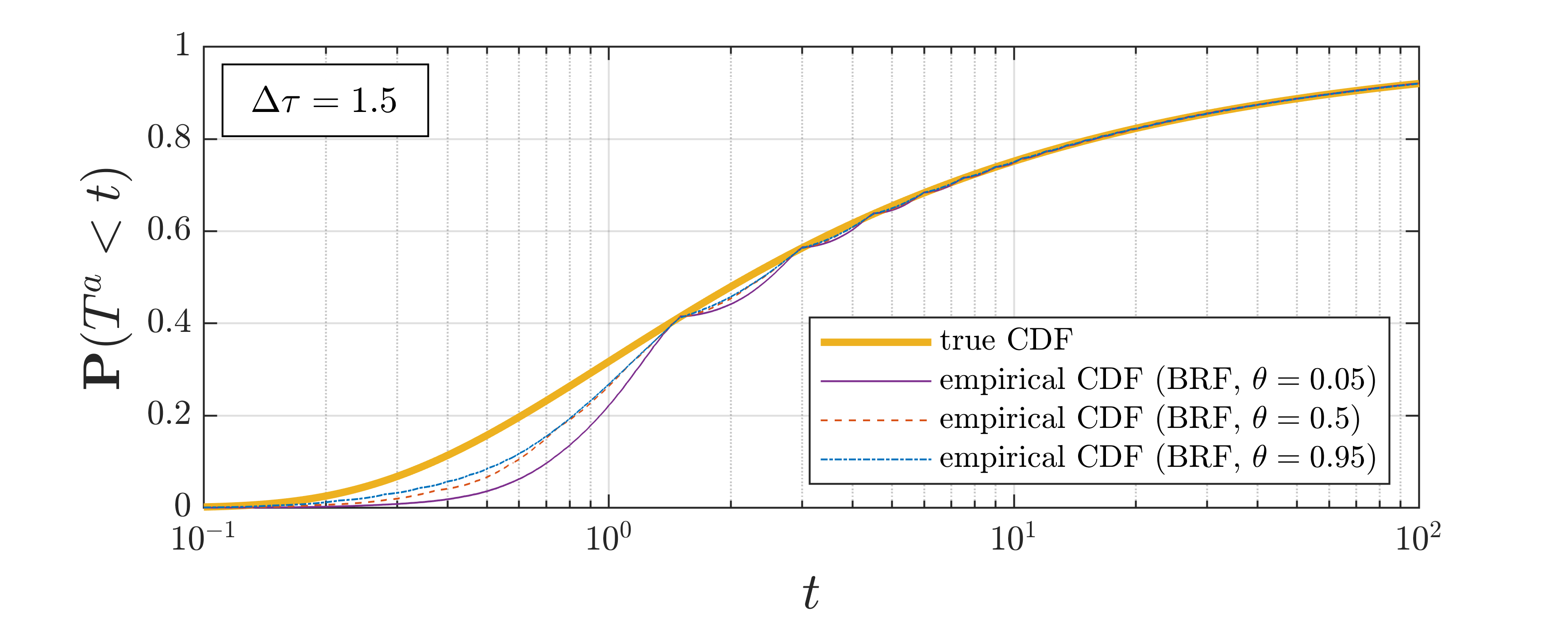}
    \caption{A comparison of Brownian root-finding with different values of $\theta$. One-dimensional Brownian motion is simulated starting at the origin with $2^{17}$ walkers and time-step $\Delta\tau = 1.5$ using the Euler-Maruyama method. Max-sampling is used to determine when a walker has passed a barrier $a = 1$, and then the BRF algorithm is run with three different parameters $\theta$ to obtain empirical exit time distributions. The empirical distribution most closely resembles the true distribution for $\theta$ close to 1, but the computation time is significantly higher for $\theta = 0.95$ (and $\theta = 0.05$) compared to $\theta = 0.5$.}
    \label{fig:BRFtest_theta}
\end{figure}

\begin{algorithm}[!htbp]
  \caption{Brownian root-finding algorithm: provides an estimate for the hitting time and location of a Brownian motion on some boundary $\partial\Omega$.}\label{alg:BRF}
  \begin{algorithmic}[1]
    \Require{Signed distance function $\dO$, consecutive steps of Brownian path $\vec{B}_{\mathrm{old}}$ and $\vec{B}_{\mathrm{new}}$, parameter $\theta \in (0,1)$, tolerance $\epsilon$, functions \ExitCondition, \tEstimate, \xEstimate}
    \Statex
    \Function{BRF}{$\vec{B}_{\mathrm{old}},\vec{B}_{\mathrm{new}},\theta,\epsilon,\dO,\ExitCondition,\tEstimate,\xEstimate$}
    \Let{$\rho_{\textrm{exit}}$}{$|\dO(\vec{B}_0)|$}
    \Let{$\tau_\textrm{min}$}{0}
    \Let{$\tau_\textrm{max}$}{$\Delta \tau$}
    \Let{$\vec{x}_\textrm{min}$}{$\vec{B}_{\mathrm{old}}$}
    \Let{$\vec{x}_\textrm{max}$}{$\vec{B}_{\mathrm{new}}$}

    \While{$\rho_{\textrm{exit}} > \epsilon$}
      \Let{$T$}{$\tau_\textrm{max} - \tau_\textrm{min}$}
      \Let{$\tau_\textrm{exit}$}{$(1-\theta)\tau_\textrm{min} + \theta\tau_\textrm{max}$}
      \Let{$t$}{$\tau_\textrm{exit} - \tau_\textrm{min}$}
      \Let{$\vec{x}_\textrm{exit}$}{$\mathcal{N}\left(\left(1 - \frac{t}{T}\right)\vec{x}_\textrm{min} + \frac{t}{T}\vec{x}_\textrm{max}, \frac{t(T - t)}{T}I_d\right)$}
      \If{$\dO(\vec{x}_\textrm{exit})>0$}
      \Let{$\tau_\textrm{max}$}{$\tau_\textrm{exit}$}
      \Let{$\vec{x}_\textrm{max}$}{$\vec{x}_\textrm{exit}$}
      \ElsIf{$\dO(\vec{x}_\textrm{exit})<0$}
      \If{$\ExitCondition\left(\vec{x}_\textrm{min},\vec{x}_\textrm{exit},t\right)$}
      \State \Return{$\tau_\mathrm{min} + \tEstimate(\vec{x}_\mathrm{min},\vec{x}_\mathrm{exit},t),\xEstimate(\vec{x}_\mathrm{min},\vec{x}_\mathrm{exit})$}
      \Else
      \Let{$\tau_\textrm{min}$}{$\tau_\textrm{exit}$}
      \Let{$\vec{x}_\textrm{min}$}{$\vec{x}_\textrm{exit}$}
      \EndIf
      \EndIf
    \Let{$\rho_{\textrm{exit}}$}{$\left|\dO(\vec{x}_\textrm{exit})\right|$}
    \EndWhile
    \State \Return{$\tau_\textrm{exit}, \vec{x}_\textrm{exit}$}
    \EndFunction
  \end{algorithmic}
\end{algorithm}

To compare Brownian root-finding algorithm to the corrected and na\"{i}ve exit time/location estimates, we once again simulate one-dimensional Brownian motion starting at the origin. We record empirical exit times past the barrier $a = 1$ using na\"{i}ve, corrected, and Brownian root-finding estimates, and compare these to the known distribution. Results obtained using the max-sampling and na\"{i}ve exit conditions for two different time-steps are shown in Fig.~\ref{fig:BRF_figure}. These results show that the BRF algorithm does outperform the corrected estimate, but not by much, since BRF relies on the corrected estimate for the case when $\vec{B}_{\mathrm{old}}, \vec{B}_{\mathrm{new}} \in \Omega$. This suggests that implementing BRF to estimate exit times and locations may not be worth the effort, since decreasing $\Delta\tau$ decreases bias more easily and effectively. Using the na\"{i}ve exit condition, the empirical and true CDFs differ significantly regardless of the exit time/location estimator. This further illustrates the need to modify the exit condition.

\begin{figure}[!htbp]
    \centering
        \includegraphics[width = 1\linewidth]{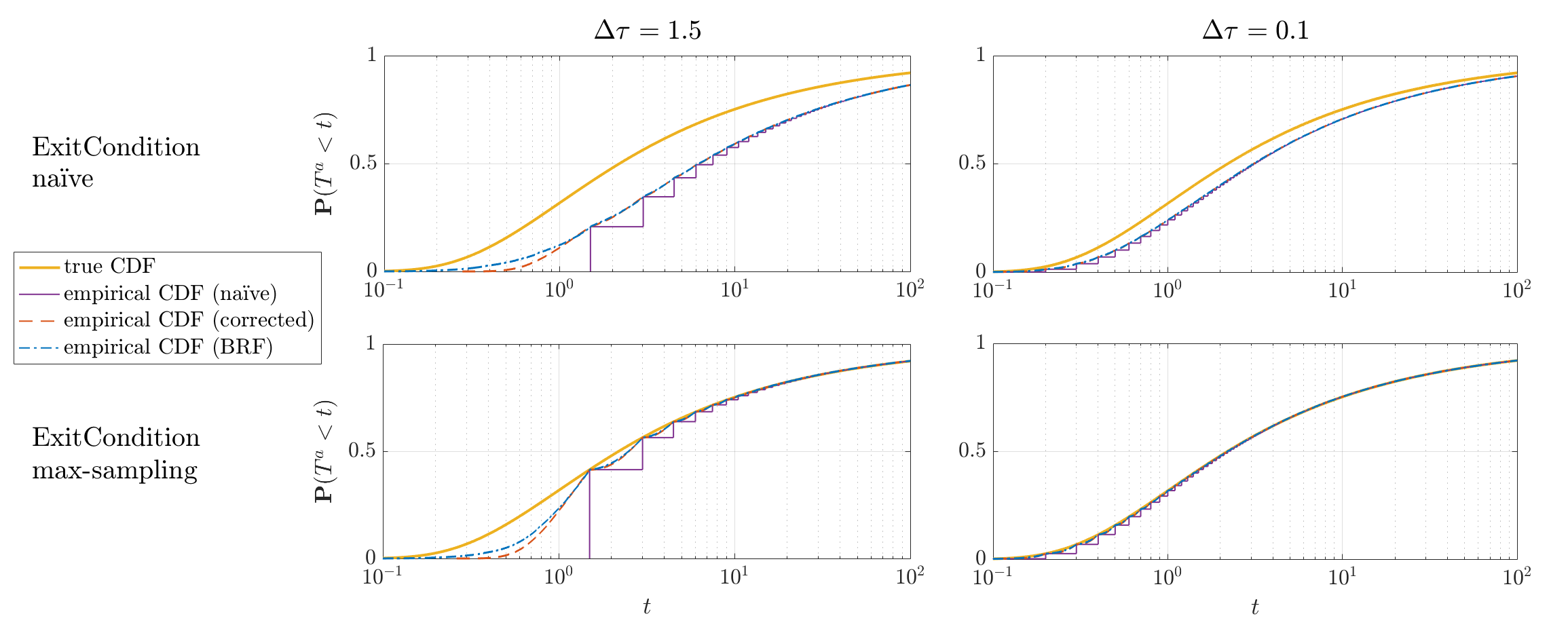}
    
    \caption{Empirical and true CDFs for the first passage time of a one-dimensional Brownian motion. The Brownian motion is simulated starting at the origin with $2^{17}$ walkers using the Euler-Maruyama method. Max-sampling is used to determine when a walker has passed a barrier $a = 1$. Three different algorithms are used to estimate the exit times: na\"{i}ve, corrected, and BRF with $\theta = 0.5$. The BRF algorithm terminates when the exit location estimate is within $\epsilon=0.01$ units of the boundary. When $\vec{B}_{\mathrm{new}} \in \Omega$, the BRF algorithm uses the corrected estimate. For large $\Delta\tau$, BRF mildly outperforms the corrected estimate in approximating the true distribution of exit times, with both vastly outperforming the na\"{i}ve estimate. For smaller $\Delta\tau$, the corrected and BRF estimates are virtually indistinguishable, both outperforming the na\"{i}ve estimate. The empirical overestimation of exit times with the na\"{i}ve exit condition (especially with a large time-step) illustrates the need for a better exit condition.}
    \label{fig:BRF_figure}
\end{figure}

\section{Temporal Difference Learning}
\label{sec:temporal_difference_learning}

In the Monte Carlo paradigm, if estimates of $u$ are desired at two or more different points $\vec{x}$, the method described above does not share information between the points. This is particularly wasteful for solutions to elliptic partial differential equations with their smooth solutions since nearby points will have nearby solution values. To be less wasteful with our samples, we propose a different paradigm to assemble information from the random samples into an estimate for a solution.

Temporal Difference Learning (TDL) is a general learning paradigm to learn the value of functionals of a Markov chain from random samples (See~\cite{Sutton1988} for its first analysis). Unlike the Monte Carlo method, which computes the value of a solution at a single point, TDL successively updates a representation of the solution at all points. 

TDL works by parametrizing the solution with parameters $c_k$ and then updating the parameters. A simple way to parametrize the solution $u$ is as a linear combination of fixed basis functions, $u_k(\vec{x})$, $k\in K$,
\begin{equation}
\label{eqn:basisrep}
\hat{u}(\vec{x}) = \sum_{k\in K} c_k u_k(\vec{x}).
\end{equation}
We look for coefficients $c_k$ so that the resulting linear combination approximates the solution to the PDE.

Note that other parametrizations of solutions $u$ are also possible. It is not necessary that the parameters $c_k$ appear linearly. All that is necessary is that the parametrized representation of the solution is able to accurately approximate within a set of functions containing our desired function. See~\cite{han_nica_stinchcombe} for an implementation of a similar method using artificial neural network parametrizations. In this approach, we trade seeking point values of the solution for determining the parameters $c_k$. This approach is reminiscent of a spectral method based on noisy data. In the machine learning literature, the functions $u_k$ are referred to as \\emph{features}.  One advantage of TDL is that any extra knowledge about the solution can be easily incorporated into the solution in the choice of the basis functions $u_k$. For example, if it is known that the solution is radially symmetric, the basis functions can all be chosen to be radially symmetric to enforce the symmetry.

\begin{algorithm}
  \caption{Temporal-Difference scheme: provides an estimate for the entire $u(\cdot)$. Several variations are possible by varying the $\ExitCondition$, $\fEstimate$ and $\gEstimate$ functions}\label{alg:TDL}
  \begin{algorithmic}[1]
    \Require{Basis function $u_k(\vec{x})$. Functions $\fEstimate,\gEstimate, \ExitCondition$. Initial parameters $\left\{c_k\right\}_{k \in K}$, learning rates $\left\{\alpha_k\right\}_{k\in K}$, number of walkers $N$.}
    \Statex
    \Function{TemporalDifference}{$n, N, \Delta \tau, \left\{c_k\right\}_{k \in K}, \left\{\alpha_k\right\}_{k\in K}$, $\ExitCondition$, $\gEstimate$,  $\fEstimate$}

    \Let{$\vec B_{\mathrm{new}}^{(i)}$}{$\mathrm{Unif}(\Omega)$}
    \Let{$\vec B_{\mathrm{old}}^{(i)}$}{$\vec B_{\mathrm{new}}^{(i)}$}
    \While{NOT ALL $\ExitCondition(\vec B_{\mathrm{old}}^{(i)},\vec B_{\mathrm{new}}^{(i)}, \Delta \tau)$}
        \For{$i \in \{1,2,\ldots,N\} \setminus \ExitCondition(\vec B_{\mathrm{old}}^{(i)},\vec B_{\mathrm{new}}^{(i)}, \Delta \tau)$}
            \Let{$\vec B_{\mathrm{old}}^{(i)}$}{$\vec B_{\mathrm{new}}^{(i)}$}
            \Let{$\vec B_{\mathrm{new}}^{(i)}$}{$\vec B_{\mathrm{old}}^{(i)} + \sqrt{\Delta \tau} \mathcal{N}(\vec{0},I_d)$}
            \For{$k \in K$}
            \Let{$c_k$}{$c_k + \alpha_k[u(\vec B_{\mathrm{new}}^{(i)})-u(\vec B_{\mathrm{old}}^{(i)}) - \frac12 \fEstimate(\vec{B}_{\mathrm{new}}^{(i)},\vec{B}_{\mathrm{old}}^{(i)},\Delta \tau)]u_k(\vec B_{\mathrm{old}}^{(i)})/N$}
            \EndFor
        \EndFor
    \EndWhile
    \For{$i \in \{1,2,\ldots,N\}$}
        \For{$k \in K$}
           \Let{$c_k$}{$c_k + \alpha_k[\gEstimate(\vec{B}_{\mathrm{new}}^{(i)},\vec{B}_{\mathrm{old}}^{(i)},\Delta \tau)-u(\vec B_{\mathrm{old}}^{(i)}) - \frac12\fEstimate(\vec{B}_{\mathrm{new}}^{(i)},\vec{B}_{\mathrm{old}}^{(i)},\Delta \tau) ]u_k(\vec B_{\mathrm{old}}^{(i)})/N$}
        \EndFor
    \EndFor
    \State \Return{$\left\{c_k\right\}_{k \in K}$}
    \EndFunction
  \end{algorithmic}
\end{algorithm}

For concreteness in the discussion below, we will use Cartesian products of Chebyshev polynomials~\citep{trefethen2013approximation} as our basis functions 
$$
u_k(x_1,x_2) = T_{k_1}(x_1) T_{k_2}(x_2),
$$
in which $k_1,k_2$ are the coordinate indices for the linear index $k$. The Chebyshev polynomials are the orthogonal polynomials $T_0(x)=1, T_1(x)=x, T_2(x)=2x^2-1, T_3(x)=4x^3-3x, \ldots $. These polynomials can provide a powerful approximation of functions and are the basis of the widely-used open-source toolbox \textit{chebfun}~\citep{driscoll2014chebfun,chebfun}.

We will employ temporal difference learning to estimate a value function $u(\vec{x})$ by combining our Monte Carlo random walk method with ideas from dynamic programming. The key aspect of a temporal difference method is that the value function is updated on each step of the method rather than only at the conclusion of the learning epoch. For our present problem, we will be able to improve our estimate of $u$ on each $\Delta \tau$ sized time-step and not just when a walker reaches the boundary.

We interpret the Feynman-Kac formula in the context of a Markov reward process. The Brownian motion process collects \emph{rewards} over time and the desired function $u(\vec{x})$ is the value of position $\vec{x}$, the expected long run rewards beginning at position $\vec{x}$. The value of a point on the boundary $\vec{x}\in\partial\Omega$, where the process terminates, is the boundary value $g(\vec{x})$. This value can be estimated as described previously with $\gEstimate$. The total reward collected by the Brownian motion is $\int_0^T f(\vec{B}(t))~\d t$ and the incremental reward over a step $\Delta \tau$ can be estimated as was done previously with $\fEstimate$. Evaluating Eq.~\eqref{eq:FeynmanKac} at $\vec x = \vec B_{\textrm{old}}$ gives
\begin{equation*}
u(\vec{B}_{\mathrm{old}}) = \E \bigg[u(\vec{B}_{\mathrm{new}}) -\frac12 \int_0^{\min(\Delta \tau,T)} f(\vec{B}(t))~\d t \bigg],
\end{equation*}
the Bellman equation for the Markov reward process. The long-run reward of the current position is the expected value of the long-run reward of the next position plus the reward of going from the current position to the next position. If the exit occurs during the time-step, i.e. $T\leq \Delta \tau$, then $u(\vec{B}_{\mathrm{new}})$ is understood to be $g(\vec{x}_\mathrm{exit})$.

A functional that is minimal when Eq.~\eqref{eq:FeynmanKac} is satisfied is
$$
J[u(\cdot)]:= \frac12 \int_\Omega \left( u(\vec x) - \tilde u(\vec x) \right)^2~\d\vec x,
$$
with \emph{target}
$$
\tilde u(\vec x) := \E \bigg[u(\vec{B}(\min(\Delta \tau,T)~|~\vec B(0)=\vec x)) -\frac12 \int_0^{\min(\Delta \tau,T)} f(\vec{B}(t))~\d t\bigg].
$$
Noting that the integrand is zero for $\vec x \in \partial \Omega$ since $T=0$, the functional derivative is
$\frac{\delta J}{\delta u} = u(\vec x) - \tilde u(\vec x).$
We minimize the functional using gradient descent at $\vec{x}=\vec{B}_{\mathrm{old}}$, i.e we update $u(\vec{B}_{\mathrm{old}})$ according to
\begin{equation}
\label{eqn:uupdate}
u(\vec{B}_{\mathrm{old}}) \leftarrow u(\vec{B}_{\mathrm{old}}) - \alpha \left( u(\vec{B}_{\mathrm{old}}) - u(\vec{B}_{\mathrm{new}}) +\frac12\fEstimate(\vec{B}_{\mathrm{old}},\vec{B}_{\mathrm{new}},\Delta \tau)  \right).
\end{equation}
This can be written as a weighted combination of an old value and a target value,
\begin{equation*}
u(\vec{B}_{\mathrm{old}}) \leftarrow (1-\alpha) u(\vec{B}_{\mathrm{old}}) + \alpha \left( u(\vec{B}_{\mathrm{new}}) -\frac12\fEstimate(\vec{B}_{\mathrm{old}},\vec{B}_{\mathrm{new}},\Delta \tau) \right).
\end{equation*}
The numerical parameter $\alpha$ is known as the learning rate and accounts for the relative confidence in the current estimate for $u$ and the newly acquired sample of the reward. In the case $f \equiv 0$, the update is simply averaging nearby values of $u$, which is what we expect for Laplace's equation.

Since the parameters $c_k$ in Eq.~\eqref{eqn:basisrep} appear linearly, updating the values of $u$ corresponds to updating the parameters as
$$
\label{eqn:cupdate}
c_k \leftarrow c_k +\alpha \left( u(\vec{B}_{\mathrm{new}}) - u(\vec{B}_{\mathrm{old}}) -\frac12\fEstimate(\vec{B}_{\mathrm{old}},\vec{B}_{\mathrm{new}},\Delta \tau) \right) u_k(\vec{B}_{\mathrm{old}}).
$$
The factor of $u_k$ is present due to the chain rule and means that the parameters having larger influence on $u$ will incur larger changes in the update. Our TDL approach is detailed in Alg.~\ref{alg:TDL}.

We test the TDL approach for the Dirichlet and Poisson problems with the na\"{i}ve and max-sampling exit conditions. The error of our approximate solutions as well as the convergence of our coefficients is shown in Fig.~\ref{fig:TDL}. The solution to the Poisson problem is in the Chebyshev basis, $\frac14(x_1^2+x_2^2-1) = \frac18(2x_1^2-1) + \frac18(2x_2^2-1)$. We choose only three members of the basis, $\hat{u}=c_{00} + c_{20} T_2(x_1)T_0(x_2)+c_{02}T_0(x_1)T_2(x_2)$, with exact coefficient values of $c_{00}=0$ and $c_{02}=c_{20}=\frac18$. TDL with the na\"{i}ve exit condition converges to a biased approximation ($c_{00}$ does not approach zero), similar to the results in Fig.~\ref{fig:bias}.

For the Dirichlet problem, we use the basis $\hat u = c_0 + c_1 \mathrm{arctan}\left(\frac{2x_2}{1-x_1^2-x_2^2}\right) + c_2T_2(x_1)T_2(x_2)$. The true solution has a discontinuity on the domain boundary, which results in a slowly converging Chebyshev series. We include the inverse tangent function in the basis so that the true solution is exactly represented with known coefficients, $c_0=1/2$, $c_1=1/\pi$, and $c_2=0$. Just as for the Poisson problem, TDL with the na\"{i}ve exit condition converges to a biased approximation ($c_1$ does not approach $1/\pi$).

In both cases, TDL resulted in an error similar to that of the Monte Carlo method but with much less computational work. TDL needed only $2^{14}$ walkers making $2^{10}$ steps~\textemdash~the resulting $2^{24}$ samples of the Brownian motion should be compared with over $2^{40}$ samples needed to obtain the results in Fig.~\ref{fig:bias}.

The efficiency of the TDL approach can be further improved by optimizing the learning rate schedule, which we have not carefully studied. Additionally, TDL, unlike the Monte Carlo method, does not require us to wait for all of the walkers to exit~\textemdash~the distribution for the \emph{last} exit has a long tail. In our method, we uniformly re-initialize any walker that exits the domain and simply terminate the algorithm after a certain number of walkers have exited.

\begin{figure}[!htbp]
    \centering
    \includegraphics[width=0.9\linewidth]{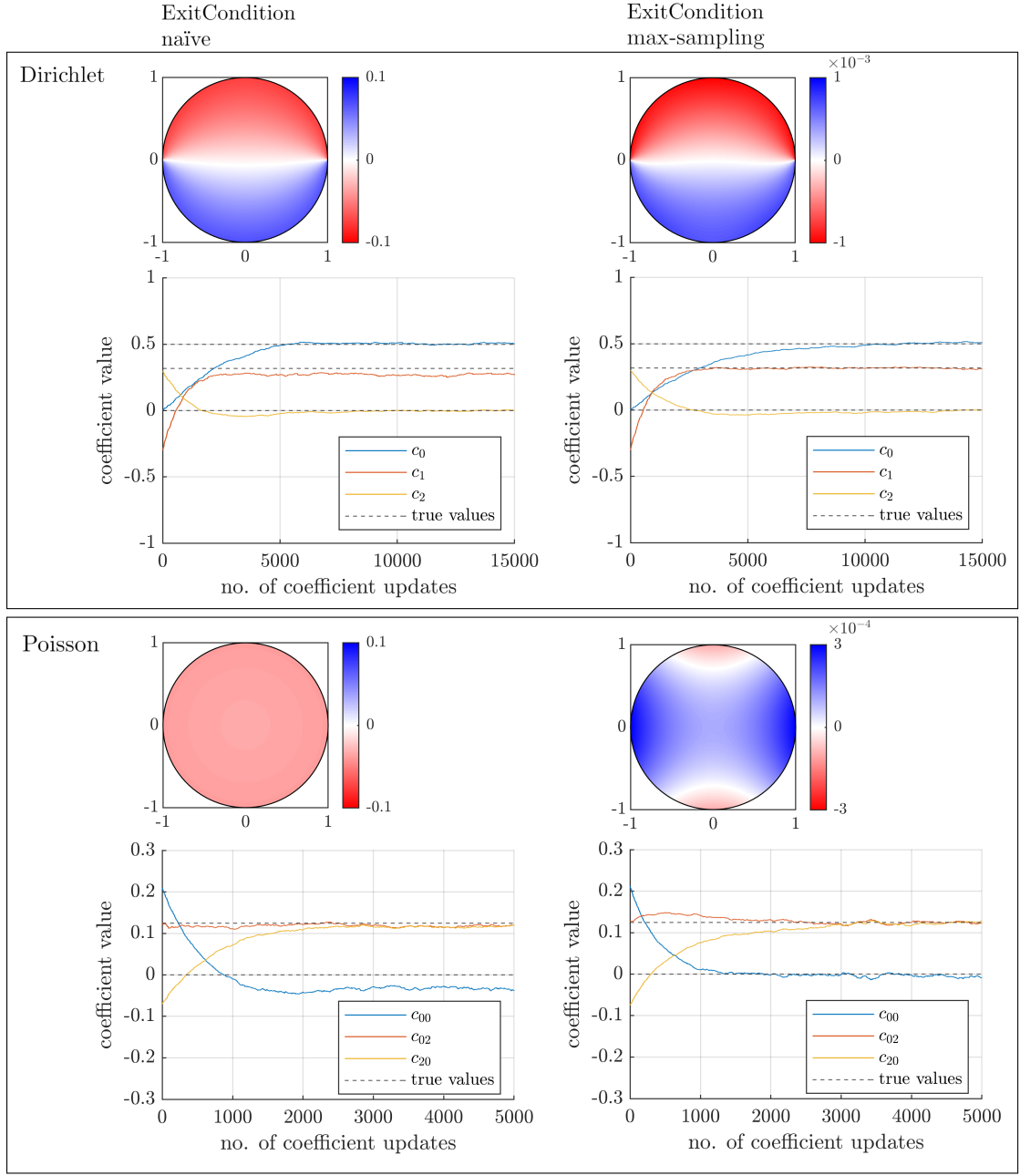}
    \caption{Temporal difference learning results for the Dirichlet and Poisson problems with the na\"{i}ve and max-sampling exit conditions. We use the corrected gEstimate and fEstimate and a time-step of $\Delta\tau = 0.01$ for both cases. The coefficient plots were obtained with $2^4$ walkers and a constant learning rate for each coefficient. For the error contour plots, we 1) use $2^{14}$ walkers, restart exited walkers, and terminate after $2^{14}$ walkers exit; 2) use a different geometrically decreasing learning rate for each coefficient, selected so that each coefficient has a similar variability. The error in all cases is similar to the bias (compare to Fig.~\ref{fig:bias}, but note its larger $\Delta \tau$), which shows that the learning has converged. The error is significantly smaller for both problems when using the max-sampling exit condition.}
    \label{fig:TDL}
\end{figure}

\section{Conclusions}

In this paper, we developed and improved a numerical method for solving elliptic quasilinear PDEs based on sampling Brownian motion. We described the traditional Monte Carlo method, identifying several of its subroutines. These subroutines rely on the accurate estimation of key quantities related to the underlying Brownian motion (e.g. exit time, exit location, local time, etc.). By improving on the na\"{i}ve subroutines, we reduce systematic bias in the Monte Carlo method. We then implemented a reinforcement learning based method to learn the solution to the PDE from sample paths of Brownian motion. We incorporated our improved subroutines into this temporal difference learning framework and showed that they improved the method's accuracy. While our examples demonstrate that our subroutines reduce bias, it should be clear that there are much better methods to solve Poisson's equation on a disk. High-dimensional problems or problems with intricate boundaries could require the Feynman-Kac formula-based approach presented here and could benefit from our bias reduction strategies. A particularly well-suited application would be options pricing with many assets~\citep{firth2005high}. Future work includes combining our sample path discretization techniques and high-performance deep learning methods to obtain competitive numerical methods for solving quasilinear elliptic PDEs.

\section*{Acknowledgements}
We gratefully acknowledge that this research was supported by the Fields Institute for Research in Mathematical Sciences. Its contents are solely the responsibility of the authors and do not necessarily represent the official views of the Institute. We acknowledge the support of the Natural Sciences and Engineering Research Council of Canada (NSERC): RGPIN-2019-06946 for ARS and PDF-502287-2017 for MN.

\bibliographystyle{econ}
\bibliography{bias}

\appendix
\section{Proof of Equation~\eqref{eqn:exit_time_ineq}} \label{A:exit_time_ineq}

We establish the following inequality about the hitting time of a Brownian bridge: Let $0<a<x$. Consider a Brownian bridge with initial position $B(0)=0$ and final position $B(\Delta \tau)=x$. Let $T_a = \inf\{t>0: B(t) > a\}$ be the first time the Brownian bridge hits a barrier at $a$. Then, we have that $\mathbb{E}[T_a]$ obeys the inequality:
\begin{equation}
 \frac{1}{1 +x^{-2}\Delta \tau }  \leq   \frac{\mathbb{E}[T_a]}{\frac{a}{x}\Delta \tau} \leq 1 \label{eqn:bb_ineq}
\end{equation} In our setting, Eq.~\eqref{eqn:exit_time_ineq} follows immediately from this fact by taking the barrier $a = |\dO(\vec B_{\mathrm{old}})|$ and the final position $x = \Delta\rho = \dO(\vec B_{\mathrm{new}})
- \dO(\vec B_{\mathrm{old}})$. 

To prove Eq.~\eqref{eqn:bb_ineq}, we use the probability density of $T_a$ from Eq.~\eqref{eqn:exittimeintensity}, to find that $\mathbb{E}{\left[T_a\right]}$ is given by
\begin{align*}
    \mathbb{E}\left[T_a\right] &= a\int_0^{\Delta\tau}\frac{\sqrt{\Delta \tau}}{\sqrt{2\pi}t^{1/2}(\Delta\tau - t)^{1/2}}\exp\left(\frac{x^2}{2\Delta \tau}-\frac{(a-x)^2}{2(\Delta\tau - t)} - \frac{a^2}{2t}\right)~\d t \\
    & = a \int_0^{\Delta\tau}\rho\left(t,a \right)~\d t  = a\mathbb{E}\left[L_a\right],
\end{align*}
where $\rho\left(t,a \right)$ denotes the probability density of the Brownian bridge to be at $B(t)=a$ at time $t$, 
and $L_a$ is the local time at $a$ of this Brownian bridge. The probability density for this local time has an explicit formula from Equation (3) in~\cite{pitman1999}, namely,
$$\mathbb{P}(L_a > y) = \exp\left(-\frac{1}{2\Delta\tau}\left((|a| + |x-a| + y)^2 - x^2\right)\right).$$
For $0<a<x$, we have $|a|+|x-a|=x$, which yields
$$
    \mathbb{E}\left[T_a\right] = a\int_0^\infty \exp\left(-\frac{1}{2\Delta\tau}\left(2xy + y^2\right)\right)~\d y.
$$

Finally, we can compute by a change of variable that
\begin{align*}
    \frac{\mathbb{E}\left[T_a\right]}{\frac{a}{x}\Delta\tau} &= \frac{x}{\Delta\tau}\int_0^\infty \exp\left(-\frac{1}{2\Delta\tau}\left(2xy + y^2\right)\right)~\d y \\
    & = \sqrt{2\pi}x\exp\left(\frac{x^2}{2\Delta\tau}\right)\int_0^\infty \frac{1}{\sqrt{2\pi}\Delta\tau}\exp\left(-\frac12\left(\frac{y + x}{\Delta\tau}\right)^2\right)~\d y \\
    & = \sqrt{2\pi}\frac{x}{\sqrt{\Delta \tau}}\exp\left(\frac{x^2}{2\Delta\tau}\right)\mathbb{P}\left(X > \frac{x}{\sqrt{\Delta\tau}}\right), \textrm{ where } X\sim \mathcal{N}(0,1).
\end{align*}

The Mill's ratio inequality from Lemma 12.9 in~\citep{mortersperes2012}, which holds for all $c>0$, gives
$$
    \frac{1}{\sqrt{2\pi}}\frac{1}{c+c^{-1}} e^{-c^{2} / 2} \leq \mathbb{P}(X > c) \leq \frac{1}{\sqrt{2\pi}} \frac{1}{c} e^{-c^{2} / 2} \label{eqn:mills}.
$$
This gives the desired result of Eq.~\eqref{eqn:bb_ineq} by setting $c = x/{\sqrt{\Delta \tau}}$.

\end{document}